\documentclass[a4paper,12pt]{article}

\usepackage{mathpazo}
\usepackage{url}
\usepackage{amsmath,amssymb, amsthm}
\usepackage{stmaryrd}
\usepackage{graphicx}
\usepackage{textcomp}
\usepackage{enumitem}
\usepackage{xspace}
\usepackage{bigints}

\usepackage[utf8]{inputenc}
\usepackage{geometry}
\usepackage[affil-it]{authblk}

\usepackage[colorlinks=true, pdfstartview=FitV, linkcolor=blue, citecolor=blue, urlcolor=blue,pagebackref=false]{hyperref}

\usepackage{xcolor}
\definecolor{darkred}{rgb}{0.9,0.1,0.1}

\newcommand{\normsob}[2]{\|#1\|_{H^k(#2)}}
\newcommand{\Xnorm}[1]{\| #1\|_{X^k}}
\newcommand{\laplace}[1]{\mathcal{L}\left(#1\right)}

\newtheorem{theorem}{Theorem}[section]
\newtheorem{definition}[theorem]{Definition}
\newtheorem{lemma}[theorem]{Lemma}

\newtheorem{proposition}[theorem]{Proposition}
\newtheorem{remark}[theorem]{Remark}

\title{ \bf On the return to equilibrium problem for axisymmetric floating structures in shallow water}
\author{\small EDOARDO BOCCHI\footnote{Institut de Mathématiques de Bordeaux UMR 5251, Université de Bordeaux,  351 Cours de la Libération, 33405 Talence, France -  \texttt{edoardo.bocchi@math.u-bordeaux.fr}}}
\date{}
	
\begin{document}
	\maketitle 
	
	\begin{abstract}
	In this paper we address the return to equilibrium problem for an axisymmetric floating structure in shallow water. First we show that the equation for the solid motion can be reduced to a delay differential equation involving an extension-trace operator whose role is to describe the influence of the fluid equations on the solid motion. It turns out that the compatibility conditions on the initial data for the return to equilibrium configuration are not satisfied, so we cannot use the result from \cite{Bocchi} for the nonlinear problem. Hence, assuming small amplitude waves, we linearize the equations in the exterior domain and we keep the nonlinear equations in the interior domain. For such configurations, the extension-trace operator can be computed explicitly and the delay term in the differential equation can be put in convolution form. The solid motion is therefore governed by a nonlinear second order integro-differential equation, whose linearization is the well-known Cummins equation. We show global in time existence and uniqueness of the solution using the conservation of the total fluid-structure energy.
	\end{abstract}

\section{Introduction}
The return to equilibrium problem is a particular configuration of the floating structure problem. It consists in releasing a partially submerged solid body in a fluid initially at rest and letting it evolve towards its equilibrium position. The interest of this problem is that it can easily be done experimentally and it is used in engineering to determine several important characteristics of floating objects. More precisely, engineers assume that the solid satisfies a linear integro-differential equation, the Cummins equation (see \cite{cummins1962}). The experimental data coming from the return to equilibrium problem (also called decay test) are then used to identify the coefficients of this linear equation.
 John in \cite{John} studied the problem in shallow water in one horizontal dimension for an object with flat bottom: he considered the linearized fluid equations for small amplitude waves and he wrote an explicit expression for the solid motion under linear approximation. Ursell in \cite{Ursell} and Maskell and Ursell \cite{MaskUrs}, using like John the linear approach, obtained an explicit solution in integral form for the vertical displacement of the object. Still under the linear approximation Cummins in \cite{cummins1962} treated a general ship motion and reduced the free motion of the floating body to an integro-differential equation. From Wehausen and Laitone \cite{WeLai} we know that also Sretenskii, several years before Cummins, obtained an integro-differential equation for the vertical displacement which he solved numerically. The Cummins equation for the vertical displacement reads
\begin{equation}
\left(m+ a_\infty \right)\ddot{\delta}_G(t)= -c\delta_G(t) -\int_{0}^{t}K(s)\dot{\delta}_G(t-s)ds,
	\label{integrodiff}
\end{equation} where $\delta_G(t)= z_G(t) -z_{G,eq}$ is the displacement from the equilibrium position of the vertical position of the centre of mass, $m$ is the mass of the structure, $a_\infty$ is the added mass at infinity frequency, $c$ is the hydrostatic coefficient and $K$ is the impulse response function (also known as retardation function and fluid memory). It appears in naval architecture and hydrodynamical engineering literature and it is used to study the motion of ships or wave energy converters. Recently Lannes in his paper \cite{Lan} on the dynamics of floating structures modelled the return to equilibrium problem using a different formulation for the hydrodynamical model with the aim to take into account nonlinear effects. He wrote the explicit equations in the one-dimensional (horizontal) case and, considering the nonlinear shallow water model, he showed that the position of the solid is fully determined by the nonlinear second order damped ODE
\begin{equation}
	(m+m_a(\delta_G)) \ddot{\delta}_G(t)= -\mathfrak{c} \delta_G(t) - \nu(\dot{\delta}_G)+ \beta(\delta_G)\dot{\delta}^2_G(t),
	\label{eqmotion1D}
\end{equation}where $m_a(\delta_G)$ is the nonlinear added mass and $\nu(\dot{\delta}_G)$ is the nonlinear damping term. Numerical simulations for the one dimensional model proposed by Lannes are made in \cite{WahlPhD}.\\
 In our recent paper \cite{Bocchi} we dealt with the two-dimensional (horizontal) case, we showed the local well-posedness for the axisymmetric floating structure problem in the shallow water regime for initial data regular enough, provided some compatibility conditions are satisfied. We considered a solid, with vertical side-walls and a cylindrical symmetry, forced to move only vertically. For such a configuration, the horizontal coordinates of the contact line between the air, the fluid and the solid, are time independent. For an object with no vertical walls, finding the horizontal coordinates of the contact line is a free boundary problem, recently solved in the one horizontal dimension case by Iguchi and Lannes in \cite{IguLan} where the contact line is replaced by two contact points. The floating structure problem for a viscous fluid in a one dimensional bounded domain is considered in \cite{MaiSMTakTuc}.\\
 The aim of this paper is to extend the work of Lannes on the return to equilibrium problem to the two-dimensional case taking into account nonlinear effects and using the same framework as in \cite{Bocchi}, which means that we consider here the axisymmetric setting, the shallow water approximation for the fluid and a solid with the properties we have stated before. An important change with respect to the one-dimensional case is the presence of delay terms in the equation governing the solid motion. The nonlinear coupled system can be treated in an abstract way but, as we show here, it requires compatibility conditions that are not satisfied in the return to equilibrium problem. For this reason, we linearize the equations in the exterior domain but we keep the nonlinear effects in the interior domain. This approach permits us to improve the classical linear model and we get a nonlinear second order delay differential equation on the vertical displacement of the structure. If we linearize around the equilibrium state we get the standard linear Cummins equation, hence we can see the result of this paper as a rigorous justification and an extension of Cummins' work.
 
\subsection{Outline of the paper}
In Section 2 we first recall the notations for the floating structures that we have used in \cite{Bocchi} and we write the equations for the coupled problem. Then we show that the differential equation for the solid motion can be written in a closed form by introducing an extension-trace operator, which takes the boundary value of the horizontal discharge, defined as the fluid horizontal velocity vertically integrated, in the exterior domain and gives the boundary value of the fluid height in the exterior domain. In Theorem \ref{nonlinearsolideq} we solve the equation by a fixed point argument. Finally we consider the return to equilibrium configuration, giving the initial conditions on the fluid and solid unknowns. It turns out that the compatibility conditions, which are necessary in order to apply the existence and uniqueness theorem from \cite{Bocchi}, are not satisfied for these particular initial conditions.\\ In Section 3 we neglect the nonlinear effects in the exterior region, but we keep them under the object provided it does not touch the bottom of the fluid domain. We write a linear-nonlinear model for the floating structure problem: we linearize the equations in the exterior domain and we keep the nonlinearities in the interior domain. Hence the equations for the fluid in the exterior domain become the linear shallow water equations and the free surface elevation in the exterior domain satisfies a wave equation. Then, by applying a Fourier-Laplace transform argument, we can write the trace of the exterior free surface elevation at the boundary as a convolution product between the inverse Laplace transform of a Hankel function and the time derivative of the displacement $\delta_G$. Hence we have that the solid motion is described by the nonlinear integro-differential equation
\begin{equation}\begin{aligned}
(m+m_a(\delta_G)) \ddot{\delta}_G=& -\mathfrak{c}\delta_G-\nu\dot{\delta}_G+\mathfrak{c} \int_{0}^{t}F(s)\dot{\delta}_G(t-s)ds +\left(\mathfrak{b}(\dot{\delta}_G)+\beta(\delta_G)\right)\dot{\delta}^2_G \ .
\end{aligned}
\label{integro}
\end{equation}Differently from \eqref{eqmotion1D}, the damping term is a linear function of $\dot{\delta}_G$ since the equations in the exterior domain are linear and it is given by a delay term due to dispersion which does not occur in one dimension. Its linearization around the equilibrium gives a reformulation of the Cummins equation for the vertical displacement \eqref{integrodiff}. We show in Theorem \ref{globalex} the global existence and uniqueness of its solution, provided an admissibility condition for the initial datum, using the conservation of the total fluid-structure energy. In Section 4 we explain the numerical method we use to plot the time evolution of the vertical displacement of the structure for the return to equilibrium problem. We compare the numerical solution to the nonlinear integro-differential equation with the solution to the linear Cummins equation and we note that for large initial data the nonlinear effects should not be neglected.
In Appendix A we define the Hankel functions and we show some properties and results.

\subsection*{Acknowledgements}
The author warmly thanks Kai Koike for his precious advises on the long time behaviour of the impulse response function and Pierre Magal for his remarks on functional differential equations with infinite delay.\\\\
The author was supported by the grants ANR-17-CE40-0025 NABUCO and ANR-18-CE40-0027 SingFlows, the Junior Chair BOLIDE of the IDEX of the University of Bordeaux, the Simone and Cino Del Duca Foundation and the Nouvelle Aquitaine Regional Council.
\section{Nonlinear floating structure equations}\label{FSequations}
Let us recall the following notations: 
$\zeta(t,r)$ is the elevation of the free surface, $h(t,r)=\zeta(t,r)+h_0$ is the fluid height, $q(t,r)$ is the horizontal discharge, \text{i.e.} the radial component of the fluid velocity vertically integrated, $\underline{P}$ is the trace of the pressure at the surface and $\zeta_w(t,r)$ is the parametrization of the bottom of the solid. The centre of mass of the solid is $G(t)=(0,0,z_G(t))$ and its velocity is $\mathbf{U}_G(t)=(0,0,w_G(t))$. We define  $\delta_G(t)=z_G(t)-z_{G,eq}$ the displacement from the equilibrium of the vertical position of the centre of mass. We denote by $\rho_m$ the density of the floating body and $H$ its height. The fluid domain is $$\Omega(t)=\{(r,z)\in \mathbb{R}_+\times \mathbb{R} \ | \ -h_0< z< \zeta(t,r) \}.$$
Moreover, as shown in Figure \ref{image1}, the presence of the solid permits us to divide the radial line in two regions, the interior domain $(0,R)$ and the exterior domain $(R,+\infty)$, whose boundary is the projection $r=R$ of the contact line between the fluid, the air and the body. Throughout all the paper we will note, for a function $f(r)$, 
\begin{equation*}
	f_i:=f_{|_{r\in(0,R)}} \ \ \ \ \ \ \ \ \mbox{and} \ \ \  \ \ \ \ \	f_e:=f_{|_{r\in(R,+\infty)}}.
\end{equation*} We have the contact constraint in the interior domain
\begin{equation}
\zeta_i(t,r)=\zeta_w(t,r).
\label{interiorconstraint}
\end{equation}As in the standard water waves theory we assume that the initial height of the fluid in the exterior domain does not vanish, \textrm{i.e.} there exists $h_{m}>0$ such that
\begin{equation}
 h_e(0,r)\geq h_{m}  \quad \mbox{for} \quad r\in(R,+\infty).
\label{he>hm}
\end{equation}For the sake of the problem, we suppose also that the solid does not touch the bottom of the domain during its motion. Hence we assume that the initial height of the fluid under the solid does not vanish, \textrm{i.e.} there exists $h_{min}>0$ such that 
\begin{equation}
h_w(0,r)\geq h_{min} \quad \mbox{for} \quad r\in(0,R),
\label{hw>hmin}
\end{equation}with $h_w(0,r)=h_i(0,r)$ in the interior domain due to \eqref{interiorconstraint}.\\ 
\begin{figure}\centering
	\includegraphics[scale=1]{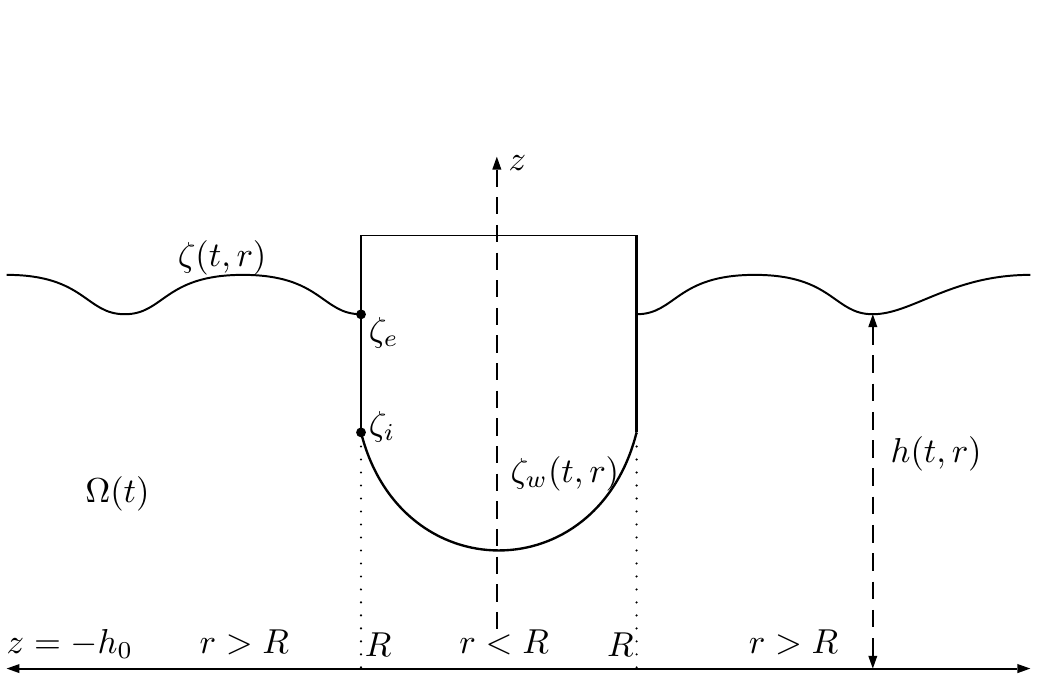}
	\caption{Vertical cross section of a cylindrically symmetric floating structure with vertical side-walls.}
	\label{image1}
\end{figure}\noindent We showed in \cite{Bocchi} that the floating structure problem in the case of an axisymmetric flow without swirl is described by 
\begin{equation}
\begin{cases}
\partial_t h +\partial_r q +\dfrac{q}{r}=0,\\[10pt]
\partial_t q + \partial_r \left(\dfrac{q^2}{h}\right)+\dfrac{q^2}{rh}+gh\partial_r h =-\dfrac{h}{\rho}\partial_r\underline{P}
\end{cases}
\label{WWFloatpolar}
\end{equation} coupled with the transition condition
\begin{equation}
q_{e_{|_{r=R}}} =q_{i_{|_{r=R}}}.
\label{neumannq}
\end{equation}
We have $\underline{P}_e=P_\mathrm{atm}$, where $P_\mathrm{atm}$ is the constant atmospheric pressure, while $\underline{P}_i$ is given by the following elliptic problem in the interior domain $(0,R)$:
\begin{gather} \begin{cases}
-\left(\partial_r + \dfrac{1}{r}\right)\left(\dfrac{h_w}{\rho}\partial_r \underline{P}_i\right)= \left(\partial_r + \dfrac{1}{r}\right)\left(\partial_r \left(\dfrac{q_i^2}{h_w}\right)+\dfrac{q_i^2}{rh_w}+gh_w\partial_r h_w\right) -\dot{w}_G, \\[10pt]
{\underline{P}_i}_{|_{r=R}}= P_\mathrm{atm}+\rho g (\zeta_e - \zeta_i)_{|_{r=R}} + P_{\mathrm{cor}},
\label{P}
\end{cases}\end{gather}with $P_{\mathrm{cor}}=\dfrac{\rho}{2}q^2_{i_{|_{r=R}}}\left(\dfrac{1}{h^2_e}_{|_{r=R}} -\dfrac{1}{h^2_i}_{|_{r=R}}\right)$. We replace $h_i=\zeta_i + h_0$ with $h_w=\zeta_w+h_0$ due to the contact constraint \eqref{interiorconstraint}. The boundary condition on the pressure is chosen in order to have exact conservation of the energy for the fluid-solid system (see \cite{Bocchi}).\\
The free motion of the solid is described by Newton's law for the conservation of the linear momentum 
 \begin{equation*}
 	m\ddot{\delta}_G(t)=-mg + \int_{r\leq R}^{}\left(\underline{P}_i
-P_\mathrm{atm}\right). \end{equation*}
 Using the elliptic equation \eqref{P} we can formulate the floating structure problem in the axisymmetric case as the following coupled problem (for details see \cite{Bocchi}):

\begin{itemize}
	\item the quasilinear hyperbolic boundary problem for the fluid motion in the exterior domain
\begin{equation}
\begin{cases}
\partial_t h_e +\partial_r q_e +\dfrac{q_e}{r}=0,\\[10pt]
\partial_t q_e + \partial_r \left(\dfrac{q_e^2}{h_e}\right)+\dfrac{q_e^2}{rh_e}+gh_e\partial_r h_e =0,\\[10pt]
{q_e}_{|_{r=R}}=-\frac{R}{2}\dot{\delta}_G,
\end{cases}
\label{systemIntro}
\end{equation}
\item  Newton's equation for the conservation of the linear momentum can be put under the form
\begin{equation}
(m+m_a(\delta_G)) \ddot{\delta}_G(t)= -\mathfrak{c} \delta_G(t) + \mathfrak{c} \zeta_e (t,R) + \left(\mathfrak{b}(h_e)+\beta(\delta_G)\right)\dot{\delta}^2_G(t)
\label{eqmotionIntro}
\end{equation} with \begin{equation} \centering \begin{aligned} &\mathfrak{c}=\rho g \pi R^2,\ \ \ \ \ \mathfrak{b}(h_e)=\dfrac{b}{h_e^2(t,R)} \ \ \ \mbox{with} \ \ \ b=\frac{\pi \rho R^4 }{8},\\[10pt]
&m_a(\delta_G)=\dfrac{\rho \pi}{2}\int_{0}^{R}\dfrac{r^3}{h_w(\delta_G,r)}dr, \\[10pt] &\beta(\delta_G)=\dfrac{b}{2 h_w^2(\delta_G,R)}+\dfrac{\pi \rho}{8}\int_{0}^{R}\dfrac{r^4}{h_w^3(\delta_G,r)}\partial_r h_w(\delta_G,r) \ dr,\end{aligned}\label{ODEparam}\end{equation}
\end{itemize}with $h_w(\delta_G,r)=h_{w,eq}(r) + \delta_G(t)$. Due to this decomposition of $h_w$ and the contact constraint \eqref{interiorconstraint}, we get the boundary condition in \eqref{systemIntro} from \eqref{neumannq} and the explicit resolution of the first equation in \eqref{WWFloatpolar} in the interior domain, that is 
\begin{equation*}
	q_i(t,r)=-\frac{r}{2}\dot{\delta}_G(t).
\end{equation*}
The term $h_{w,eq}(r)$ is the fluid height under the solid at the equilibrium position and $\zeta_{w,eq}(r)=h_{w,eq}(r)-h_0$ is the elevation of the bottom of the solid at the equilibrium position. They both depend on the density of the fluid $\rho$, the density of the solid $\rho_m$, the depth $h_0$ and the height of the solid $H$ (see Section 3 for the explicit expressions in the flat bottom case).

\subsection{Extension-trace operator for the coupling with the exterior domain}
In this section we want to show that, in the ODE for the solid part of the coupled system \eqref{systemIntro} - \eqref{eqmotionIntro}, we can write the coupling term $\zeta_e(t,R)$ (also $h_e^2(t,R)$), the trace of the free surface elevation in the exterior domain at the boundary $r=R$, as an extension-trace operator applied to the trace of the horizontal discharge in the interior domain at the boundary $r=R$, that is $-\frac{R}{2}\dot{\delta}_G$.\\ 
We consider the exterior quasilinear hyperbolic initial boundary value problem \eqref{systemIntro} and using $u=(\zeta_e,q_e)^T$ we can write it as 
\begin{equation}
\begin{cases}
	\partial_t u +A(u)\partial_r u +B(u,r)u=0,\\
	{q_e}_{|_{r=R}}=-\frac{R}{2}\dot{\delta}_G,\\
	u(0)=u_0,
\end{cases}
\label{quasilinearU}
\end{equation} with \begin{equation*}A(u)=\left (
\begin{array}{cc}
0& 1 \\
gh_e-\dfrac{q_e^2}{h_e^2}& \dfrac{2q_e}{h_e}\\
\end{array}
\right ), \ \ \ \  B(u,r)=\left (
\begin{array}{cc}
0& \dfrac{1}{r} \\[10pt]
0& \dfrac{q_e}{rh_e},\\
\end{array}
\right )\end{equation*} and $$u_0=(\zeta_{e,0},q_{e,0})^T.$$
We consider the functional space\begin{equation*}
X^k(T):=\bigcap_{j=0}^k C^j([0,T], H_r^{k-j}((R,+\infty)))$$ endowed with the norm $$\| u\|_{X^k(T)}:=\sup_{[0,T]}\Xnorm{u(t)} \ \ \ , \ \ \  \Xnorm{u(t)}=\sum_{j=0}^{k}\|\partial_t^j u(t)\|_{H_r^{k-j}((R,+\infty))},
\end{equation*}where $H^k_r:=H^k(rdr)$ is the weighted Sobolev space.
In Theorem 5.3 of \cite{Bocchi} we showed that, for $k\geq 2$, there exists $T>0$ and a unique solution $u=(\zeta_e,q_e)^T\in X^k(T)$ to \eqref{quasilinearU}, provided the initial data $u_0\in H^k_r((R,+\infty))$, the boundary datum ${q_e}_{|_{r=R}} \in H^k((0,T))$ and compatibility conditions are satisfied up to order $k-1$. Moreover $u$ satisfies the following energy estimate:
\begin{equation}
\begin{aligned} 
\Xnorm{u(t)}^2 + \normsob{u_{|_{r=R}}}{(0,t)}^2 \leq C\left(T,\|u_0\|^2_{H^k_r((R,\infty))}, \normsob{{q_e}_{|_{r=R}}}{(0,t)}^2\right)
\end{aligned}
\label{aprioriestimateHs}
\end{equation}
for all  $t\in (0,T)$. Then we can define an operator $\mathcal{B}$ such that 
\begin{equation}
\begin{aligned}
\mathcal{B}: \ \ \ H^k((0,T))\times H^k_r((R,\infty))\ \ \ &\rightarrow \ \ \ H^k((0,T))\\
      ( \dot{\delta}_G \ \ \ \ \ ,\ \ \ \ u_0)\ \  \quad\quad \quad  &\mapsto \ \ \  \mathcal{B}\left[ \dot{\delta}_G, u_0\right]={\zeta_e}_{|_{r=R}}.
\end{aligned}
\label{extensiontrace}
\end{equation} We call it an extension-trace operator since it takes the trace of $q_e$, that is $-\frac{R}{2}\dot{\delta}_G$, the initial data $u_0$ and it extends to the couple $(\zeta_e,q_e)$ by solving the initial boundary value problem \eqref{quasilinearU} and then it takes the trace of $\zeta_e$. One can easily note that $\mathcal{B}$ is nonlinear. Then, using the fact that $h_e=\zeta_e+h_0$ and assuming $u_0$ to be given, we can write the equation \eqref{eqmotionIntro} for the solid motion as a second order delay differential equation only in terms of $\delta_G$, namely
\begin{equation}\begin{aligned}
&(m+m_a(\delta_G)) \ddot{\delta}_G(t)\\&= -\mathfrak{c} \delta_G(t) + \mathfrak{c} \mathcal{B}\left[\dot{\delta}_G,u_0\right](t)+ \left(\frac{b}{(\mathcal{B}\left[\dot{\delta}_G,u_0\right](t)+h_0)^2}+\beta(\delta_G)\right)\dot{\delta}^2_G(t).
\end{aligned}
\label{eqextensiontrace}
\end{equation}
 It is a delay differential equation since we need to know $\dot{\delta}_G$ for all $t^\prime\in[0,t]$ in order to know the value of $\mathcal{B}\left[\dot{\delta}_G,u_0\right]$ at time $t$.
 This equation can be solved by a standard fixed point argument. Let us  recall the compatibility conditions on the initial data (see \cite{Schochet}):
 \begin{definition}\label{compcondFS}
 	The data $u_0\in H_r^k((R,+\infty))$, $\delta_0\in\mathbb{R}$ and $\delta_1\in\mathbb{R}$ of the floating structure coupled system \eqref{quasilinearU} - \eqref{eqextensiontrace} satisfy the compatibility conditions up to order $ k-1$ if, for $0\leq j\leq k-1$, the following holds:
\begin{equation*}
\mathbf{e}_2\cdot \mbox{\textquotedblleft} \partial_t^j u_{|_{t=0}}\textquotedblright_{|_{r=R}}= -\frac{R}{2} \mbox{\textquotedblleft} \frac{d^{j+1}}{dt^{j+1}}{\delta_G}_{|_{t=0}}\textquotedblright,
\end{equation*}
where $\mbox{\textquotedblleft} u_{|_{t=0}}\textquotedblright=u_0$, $\mbox{\textquotedblleft} \frac{d}{dt}{\delta_G}_{|_{t=0}}\textquotedblright=\delta_1$, and for $j\geq 1$  $\mbox{\textquotedblleft} \partial_t^j u_{|_{t=0}}\textquotedblright$ and $\mbox{\textquotedblleft} \frac{d^{j+1}}{dt^{j+1}}{\delta_G}_{|_{t=0}}\textquotedblright$ are inductively defined by formally taking $j-1$ time derivatives of system \eqref{quasilinearU} and of \eqref{eqextensiontrace} respectively, and evaluating at $t=0$. For instance, 
\begin{equation*}
	\mbox{\textquotedblleft} \partial_t^1 u_{|_{t=0}}\textquotedblright=-A(u_0)\partial_r u_0- B(u_0,r)u_0,
\end{equation*} 
\begin{equation*}
		\mbox{\textquotedblleft}\frac{d^{2}}{dt^{2}}{\delta_G}_{|_{t=0}}\textquotedblright= \frac{-\mathfrak{c} \delta_0 + \mathfrak{c} \mathcal{B}\left[\dot{\delta}_G,u_0\right]_{|_{t=0}}+ \bigg(\frac{b}{(\mathcal{B}\left[\dot{\delta}_G,u_0\right]_{|_{t=0}}+h_0)^2}+\beta(\delta_0)\bigg)\delta_1^2}{m+m_a(\delta_0)}.
\end{equation*}
 \end{definition}
 \vspace{0.3cm}
Then, we can state the following existence result whose proof (in the case of a solid with a flat bottom) is detailed by the author in \cite{Bocchi} (Theorem 5.3):
\begin{theorem}\label{nonlinearsolideq}
 For $k\geq 2$, let $u_0=(\zeta_{e,0}, q_{e,0})$, $ \delta_0$ and $\delta_1$ satisfy the compatibility conditions in Definition \ref{compcondFS} up to order $k-1$. Assume that there exist some constants $h_\mathrm{m}, c_\mathrm{sub}>0$ such that 
 $$\mbox{for}\quad r\in(R,+\infty) \qquad h_{e,0}(r)\geq h_\mathrm{m}, \quad \left(gh_{e,0} - \frac{q^2_{e,0}}{h^2_{e,0}}\right)(r)\geq c_\mathrm{sub}, $$with $h_{e,0}=h_0 +\zeta_{e,0},$ and that $$\delta_0 > -\inf\limits_{(0,R)}h_{w,eq}.$$Then, there exists $T>0$ such that the Cauchy problem for \eqref{eqextensiontrace} with initial data $$\delta_G(0)=\delta_0, \ \ \ \ \ \ \ \ \dot{\delta}_G(0)=\delta_1,$$ admits a unique solution $\delta_G\in H^{k+1}((0,T))$.
 \end{theorem}
\subsection{The return to equilibrium configuration}
We want to focus now on a particular configuration of the floating structure problem, the return to equilibrium problem. It consists in dropping the solid, with no initial velocity, into a fluid initially at rest from a non-equilibrium position.
By the definition of this particular configuration, we have specific initial conditions for the coupled problem \eqref{systemIntro} - \eqref{eqmotionIntro}.\\
The initial conditions for the solid equation are $$\delta_G(0)=\delta_0\neq0, \ \ \ \ \ \ \ \ \ \ \  \ \dot{\delta}_G(0)=\delta_1=0,$$ and for the fluid equations are $$h_e(0,r)=h_0, \ \ \ \ \ \ \ \ \ q_e(0,r)=0,$$ for all $r\in(R,+\infty)$. In order to apply the  theory of the initial boundary value problem we need these specific initial data to satisfy the compatibility conditions defined in \cite{Bocchi}. The compatibility conditions of order 0 and 1 are respectively:
\begin{equation*}\begin{aligned}
\bullet &\ q_e(0,R)=-\frac{R}{2}\delta_1,\\\\
\bullet & -\partial_r \left(\frac{q^2_e}{h_e}\right)(0,R)-\frac{1}{R}\frac{q^2_e}{h_e}(0,R)-gh_e(0,R)\partial_r\zeta_e(0,R)\\&=
-\dfrac{R}{2\left(m+m_a(\delta_0)\right)}\left(-\mathfrak{c}\delta_0 + \mathfrak{c}\zeta_e(0,R)+\left(\frac{\mathfrak{b}}{h_e^2(0,R)} +\beta(\delta_0)\right)\delta_1^2\right).
\end{aligned}
\label{1e2comp}
\end{equation*}
Due to the nature of the return to equilibrium configuration, we have 
\begin{equation}
\partial_r \zeta_e(0,R)=0, \ \ \ \ \zeta_e(0,R)=0, \ \ \ \ q_e(0,R)=0.
\label{returnconditions}
\end{equation}Therefore the compatibility condition of order 0 is satisfied but not the one of order 1. 
Then Theorem 5.3 of $\cite{Bocchi}$ cannot be applied since one hypothesis required is that the initial and boundary data must satisfy the compatibility conditions at least up to order $1$. When the compatibility conditions at order 1 are not satisfied, sonic waves propagate (we refer to Métivier \cite{Met91} for the existence of such waves).
\begin{remark}
One can choose a different value for $\delta_1$ in order to satisfy the compatibility conditions and be able to apply the results of Theorem 5.3 in $\cite{Bocchi}$.
\end{remark}
\section{Linear-nonlinear model for floating structures }
The impossibility to apply the theory of initial boundary value problems to the particular configuration of the return to equilibrium brings us to consider a linearization of the equations \eqref{WWFloatpolar} in the exterior domain, which describes the case of small amplitude waves. We generalize however the works by Cummins and other authors in the literature by keeping the nonlinear effects in the interior domain. We only assume that the solid does not touch the bottom of the fluid domain. In this section we introduce the linear-nonlinear model for the floating structure problem, we prove the conservation of the total energy for this model and then we show that with this linear approximation we can write the extension-trace operator $\mathcal{B}[\dot{\delta}_G,u_0]$ (simply written $\mathcal{B}[\dot{\delta}_G]$ from now on) as a linear convolution operator. Then the delay differential equation \eqref{eqextensiontrace} for the solid motion becomes a nonlinear second order integro-differential equation.
\subsection{An energy conserving linear-nonlinear model}
We consider the following linear-nonlinear model for the floating structure problem: 
\begin{itemize}
\item in the exterior domain $(R,+\infty)$ 
 \begin{equation}
 \begin{cases}
 \partial_t \zeta_e +\partial_r q_e +\dfrac{q_e}{r}=0\\[10pt]
 \partial_t q_e + gh_0 \partial_r \zeta_e =0
 \end{cases}
 \label{linearmodel}
 \end{equation}
 \item in the interior domain $(0,R)$
 \begin{equation}
 \begin{cases}
 \partial_t h_i +\partial_r q_i +\dfrac{q_i}{r}=0\\[10pt]
 \partial_t q_i + \partial_r \left(\dfrac{q_i^2}{h_i}\right)+\dfrac{q_i^2}{rh_i}+gh_i\partial_r h_i =-\dfrac{h_i}{\rho}\partial_r\underline{P}_i
 \end{cases}
 \label{nonlinearmodel}
 \end{equation}
\end{itemize}
 and the boundary conditions
 \begin{equation}
 	{q_e}_{|_{r=R}} ={q_i}_{|_{r=R}}
 	\label{asymtransition0}
 \end{equation}
 \begin{equation}
 	{\underline{P}_i}_{|_{r=R}}= P_\mathrm{atm}+\rho g (\zeta_e - \zeta_i)_{|_{r=R}} + P_{\mathrm{cor}}
 	\label{pcor}
 \end{equation}with $P_{\mathrm{cor}}=-\dfrac{\rho}{2}{\dfrac{q_i^2}{h_i^2}}_{|_{r=R}}$.
  As in the full nonlinear case the condition \eqref{asymtransition0} can be written in terms of the solid vertical displacement $\delta_G$ and it becomes 
 \begin{equation}
 {q_e}_{|_{r=R}} =-\frac{R}{2}\dot{\delta}_G.
 \label{asymtransition}
 \end{equation}
 Furthermore we have the conservation of the energy for the new linear-nonlinear model (see \cite{Bocchi} for the conservation of the energy in the full nonlinear model):
\begin{proposition}Let us define the shallow water fluid energy for the linear-nonlinear shallow water equations \eqref{linearmodel} - \eqref{nonlinearmodel}
	\begin{equation}E_{SW}= 2\pi\frac{\rho}{2}g\int_{0}^{+\infty}{\zeta}^2rdr +2\pi\frac{\rho}{2}\int_{0}^{R}\frac{{q_i}^2}{h_i}rdr+2\pi\frac{\rho}{2}\int_{R}^{+\infty}\frac{{q_e}^2}{h_0}rdr\end{equation}
	and the solid energy (only with vertical motion)
	$$E_{sol}=\frac{1}{2}mw_G^2+mgz_G.$$
	\label{conservationenergy} Then the total fluid-structure energy
	$E_{tot}=E_{SW}+E_{sol}$ is conserved, \textit{i.e.} $$\frac{d}{dt}E_{tot}= 0.$$
\end{proposition}
\begin{proof}
	Multiplying the first equation of \eqref{linearmodel} by $\rho g \zeta_e r$,  the second equation by $\dfrac{\rho q_e}{h_0}r$ and summing them together, we have local conservation of the energy, \textit{i.e.}
	\begin{equation}
	\partial_t e_{ext} + \partial_r F_{ext}=0,
	\label{localexteriorenergy}
	\end{equation}
	where $e_{ext}$ is the local fluid energy in the exterior domain $$e_{ext}=\frac{\rho}{2}g\zeta_e^2 r + \frac{\rho}{2}\frac{q_e^2}{h_0}r$$ and $F_{ext}$ is the flux in the exterior domain
	$$F_{ext}=\rho g\zeta_e q_e r.$$ 
	We consider the equations \eqref{nonlinearmodel} in the interior domain. Analogously,
multiplying the first equation by $\rho g \zeta_i r$, the second equation by $\dfrac{\rho q_i}{h_i}r$ and summing them together, 
we obtain
	\begin{equation}
	\partial_t e_{int} +\partial_r F_{int}= - r q_i \partial_r\underline{P}_i , 
	\label{localinteriorenergy}
	\end{equation}where $e_{int}$ is the local fluid energy in the interior domain $$e_{int}=\frac{\rho}{2}g\zeta_i^2 r + \frac{\rho}{2}\frac{q_i^2}{h_i}r$$ and $F_{int}$ is the flux in the interior domain
	$$F_{int}=\dfrac{\rho q_i^3}{2h_i^2}r+\rho g\zeta_i q_i r.$$For the sake of clarity we remark that, in order to derive $\partial_t e_{int} $ in \eqref{localinteriorenergy}, we have used the following identity for the convective term:
	 \begin{align*}\left(\partial_r \left(\frac{q_i^2}{h_i}\right)+ \frac{q_i^2}{r h_i}\right)\left(\frac{\rho q_i}{h_i}r\right)&= \frac{\rho q_i^2}{2h_i^2}r \left(\partial_r q_i + \frac{q_i}{r}\right)+ \partial_r \left(\frac{\rho q_i^3}{2h_i^2}r\right)\\&=-\frac{\rho q_i^2}{2h_i^2}r\partial_t \zeta_i +\partial_r \left(\frac{\rho q_i^3}{2h_i^2}r\right), \end{align*}where the last equality is due to the first equation in \eqref{nonlinearmodel} and to the fact that $\partial_t h_i=\partial_t \zeta_i$.
	We integrate \eqref{localexteriorenergy} on $[R, +\infty)$ and \eqref{localinteriorenergy} on $[0,R]$ and multiplying them by $2\pi$ we obtain 
	\begin{equation}
	\frac{d}{dt} E_{SW} - 2\pi\rho R g \left\llbracket \zeta q \right\rrbracket + 2\pi \rho R {\dfrac{q_i^3}{2h_i^2}}_{|_{r=R}}= -2\pi\int_{0}^{R}rq_i\partial_r\left(\underline{P}_i -P_\mathrm{atm}\right)dr, 
	\end{equation}where $\left\llbracket f \right\rrbracket$ is the jump of a function $f$ at the boundary $r=R$ defined as
	$$\left\llbracket f \right\rrbracket:= {f_e}_{|_{r=R}} -{f_i}_{|_{r=R}}.$$By integration by parts we get 
	\begin{equation}\begin{aligned}
	\frac{d}{dt} E_{SW} = & \ 2\pi\rho R g\left\llbracket \zeta q\right\rrbracket -2\pi \rho R {\dfrac{q_i^3}{2h_i^2}}_{|_{r=R}} -2\pi R \left(\underline{P}_i -P_\mathrm{atm}\right)_{|_{r=R}}q_{i_{|_{r=R}}}\\&+2\pi\int_{0}^{R}\left(\underline{P}_i -P_\mathrm{atm}\right)\partial_r(rq_i)dr.
	\end{aligned}
	\end{equation}
	On the other hand, from the definition of $E_{sol}$, we have 
	\begin{equation*}
	\begin{aligned}
	\frac{d}{dt}E_{sol}&= mw_G\dot{w}_G + mgw_G= w_G\left(m\dot{w}_G + mg\right)\\&=w_G \ 2\pi\int_{0}^{R}\left(\underline{P}_i -P_\mathrm{atm}\right)rdr\\&=2\pi\int_{0}^{R}\left(\underline{P}_i -P_\mathrm{atm}\right)\partial_t\zeta_wrdr,
	\end{aligned}
	\end{equation*} where we used Newton's law for the conservation of the linear momentum and, since the structure moves only vertically,  
	$$\partial_t \zeta_w = w_G,$$ coming from standard solid mechanics. From the contact constraint \eqref{interiorconstraint} and the mass conservation equation in \eqref{nonlinearmodel} we get
	\begin{equation}
	\frac{d}{dt}E_{sol}=-2\pi\int_{0}^{R}\left(\underline{P}_i -P_\mathrm{atm}\right)\partial_r(rq_i )dr.
	\end{equation}Therefore 
	$$\frac{d}{dt}E_{SW}= -\frac{d}{dt}E_{sol} +2\pi\rho R g\left\llbracket \zeta q\right\rrbracket -2\pi \rho R {\dfrac{q_i^3}{2h_i^2}}_{|_{r=R}} -2\pi R \left(\underline{P}_i -P_\mathrm{atm}\right)_{|_{r=R}}q_{i_{|_{r=R}}}.$$	
	Using the expression of the interior pressure $\underline{P}_i$ on the boundary $r=R$ in \eqref{pcor} and the transition condition \eqref{asymtransition0} we get the conservation of the total energy. 
\end{proof} 

 \subsection{Linear equations in the exterior domain} 
In this subsection we focus on the linear shallow water equations in the exterior domain
\begin{equation}
\begin{cases}
\partial_t \zeta_e +\partial_r q_e +\dfrac{q_e}{r}=0,\\[10pt]
\partial_t q_e +v_0^2\partial_r \zeta_e =0,\\
\end{cases}
\label{extlinshallow}
\end{equation}with $v_0= \sqrt{gh_0}$, coupled with the transition condition
\begin{equation}
{q_e}_{|_{r=R}}=-\frac{R}{2}\dot{\delta}_G(t).
\label{lintransition}
\end{equation}
\noindent Taking the derivative of the first equation in \eqref{extlinshallow} with respect to time and replacing the value of $\partial_t q_e$ with the expression in the second equation we find the linear wave equation
$$\partial_{tt} \zeta_e -v_0\Delta_r \zeta_e=0$$ with $\Delta_r:=\partial_{rr} +\dfrac{1}{r}\partial_r$.\\
We consider only positive time $t$ (we can treat $\zeta_e$ as a causal function, i.e. $\zeta_e=0$ for $t <0$). In the same way as John did in \cite{John}, we apply the Laplace transform 
$$\laplace{\zeta_e}(r,s)=\int_{0}^{+\infty}\zeta_e(t,r)e^{-st}dt \quad  \mbox{for}\quad \Re(s)>0 $$
 to the wave equation and we get the following Helmholtz equation with complex coefficients:
\begin{equation}s^2\laplace{\zeta_e} - v_0\Delta_r \laplace{\zeta_e}=0.\label{helmholtz}\end{equation} 
We have $\laplace{\partial_{tt}\zeta_e}=s^2\laplace{\zeta_e}+\partial_t \zeta_e(0) +s\zeta_e(0)$ but in this configuration we have in addition $\partial_t \zeta_e(0)=0$ and $\zeta_e(0)=0$ from \eqref{returnconditions}.
The general solution of \eqref{helmholtz} is $$\laplace{\zeta_e}(r,s)=a_1(s)H_0^{(1)}\left(\frac{isr}{v_0}\right)+ a_2(s)H_0^{(2)}\left(\frac{isr}{v_0}\right),$$ where $H_0^{(1)}$ and $H_0^{(2)}$ are the Hankel functions of first order and second order respectively with index 0. 

\begin{remark}
Let us consider the Bessel functions of the first kind and of the second kind, respectively $J_n$ and $Y_n$, solutions to 
\[z^{2}\dfrac{{\mathrm{d}}^{2}w}{{\mathrm{d}z}^{2}}+z\dfrac{\mathrm{d}w}{\mathrm{d%
	}z}+(z^{2}-n^{2})w=0, \ \ \ z\in\mathbb{C}.\]
The Hankel functions of first order with index $n$ are defined as
\begin{equation*}
	H_n^{(1)}= J_n + i Y_n,
\end{equation*}and the Hankel functions of second order with index $n$ as
\begin{equation*}
H_n^{(2)}= J_n - i Y_n.
\end{equation*}
\end{remark}

\noindent From the asymptotic behaviour of the Hankel functions (see Appendix A) we know that $H_0^{(1)}(z)\sim \sqrt{\dfrac{2}{\pi z}}e^{iz}$ and $H_0^{(2)}(z)\sim \sqrt{\dfrac{2}{\pi z}}e^{-iz}$ for large $|z|$ and $0<\arg z <\pi.$ Therefore for large $|s|r$ and $-\dfrac{\pi}{2}<\arg s < \dfrac{\pi}{2}$
$$H_0^{(1)}\left(\frac{isr}{v_0}\right)\sim \sqrt{\dfrac{2v_0}{\pi isr}}e^{\tfrac{-sr}{v_0}},$$ $$H_0^{(2)}\left(\dfrac{isr}{v_0}\right)\sim \sqrt{\dfrac{2v_0}{\pi isr}}e^{\tfrac{sr}{v_0}}.$$ Thus for $\Re (s)>0$ and large $r$ $$a_1(s)H_0^{(1)}\left(\frac{isr}{v_0}\right)e^{st}\sim a_1(s) \sqrt{\dfrac{2v_0}{\pi isr}}e^{s\left(t-\tfrac{r}{v_0}\right)},$$$$a_2(s)H_0^{(2)}\left(\frac{isr}{v_0}\right)e^{st}\sim a_2(s) \sqrt{\dfrac{2v_0}{\pi isr}}e^{s\left(t+\tfrac{r}{v_0}\right)}.$$These terms represent respectively an outgoing progressive wave and an incoming progressive wave. Since in this problem we consider only outgoing waves, we impose $a_2(s)=0$. \\
Applying the Laplace transform to the second equation of \eqref{extlinshallow}, we get the following boundary condition for the exterior Helmholtz problem:
$$\partial_r \laplace{\zeta_e}_{|_{r=R}}= -\frac{s}{v_0^2}\laplace{q_e}_{|_{r=R}}=\dfrac{sR}{2v_0^2}\laplace{\dot{\delta}_G},$$ 
using the transition condition \eqref{lintransition}. Therefore we finally have
\begin{equation} \laplace{\zeta_e}(s,R)=\dfrac{iR H_0^{(1)}\left(\dfrac{is R}{v_0}\right)}{2v_0H_1^{(1)}\left(\dfrac{is R}{v_0}\right)}\laplace{\dot{\delta}_G}(s),
\label{zetalaplace}
\end{equation} using the relation $(H_0^{(1)})^\prime= - H_1^{(1)}$ between the derivative of $H_0^{(1)}$ and the Hankel function of first order with index 1.
From Appendix A we have $\dfrac{H_0^{(1)}(s)}{H_1^{(1)}(s)}\rightarrow i$ for large $|s|$. Adding and subtracting this limit we have
\begin{equation}
	\laplace{\zeta_e}(s,R)= f(s)\laplace{\dot{\delta}_G}(s) -\dfrac{R}{2v_0}\laplace{\dot{\delta}_G}(s)
	\label{zetahat}
\end{equation}
with $$f(s)= \dfrac{iR H_0^{(1)}\left(\dfrac{is R}{v_0}\right)}{2v_0H_1^{(1)}\left(\dfrac{is R}{v_0}\right)}+\dfrac{R}{2v_0}$$
with $f(s) \rightarrow 0$ as $|s| \rightarrow+\infty$. It turns out that we can write $f$ as a Laplace transform of some function: 
\begin{lemma}\label{F}
There exists a unique function $F\in L^2\left(\mathbb{R}_+\right)\cap C\left([0,+\infty)\right)$ such that $f(s)= \laplace{F}(s)$, with 
\begin{equation*}
	F(t)= \lim\limits_{v\rightarrow +\infty}\frac{1}{2\pi}\int_{-v}^{v}f(c+i\omega)e^{(c+i\omega) t}d\omega,
\end{equation*}
independent of $c>0$, in the sense of $L^2$-Fourier transform and
\begin{equation*}
F(t)= \frac{1}{2\pi}\int_{-\infty}^{+\infty}\left[f(c+i\omega)- \frac{\lambda}{c+i\omega}\right]e^{(c+i\omega) t}d\omega + \lambda,
\end{equation*}
with $\lambda=\frac{1}{4}$, in the sense of Lebesgue integral.
\end{lemma}
\begin{proof} We know that both $H_0^{(1)}(is)$, $H_1^{(1)}(is)$ are holomorphic functions on $\mathbb{C}_+$, and $H_1^{(1)}(is)\neq0$ in $\mathbb{C}_+$ (see \cite{AbraStegun},\cite{NIST:DLMF}), then $f(s)$ is holomorphic on $\mathbb{C}_+$. Moreover $f$ is bounded in $\mathbb{C}_+$ since  $f\rightarrow 0$ at infinity and $f$ is bounded around the boundary $i\mathbb{R}$ (from Appendix A we have $\frac{H_0^{(1)}(is)}{H_1^{(1)}(is)} \sim -is\log(is)$ for $s\rightarrow 0$). Hence $f\in H^\infty(\mathbb{C}_+)$. Now we want to show that $f\in L^2(i\mathbb{R})$: $f$ is defined 
also in $\overline{\mathbb{C}_+}$ if we consider the one-valued functions $H_0^{(1)}$ and $H_1^{(1)}$ (considering the one-valued logarithm in the definition of the Hankel functions in  Appendix A). Moreover we have that \begin{equation}f(s)=\dfrac{1}{4s}+ O\left(\dfrac{1}{s^2}\right)\label{infinitybeh}\end{equation} as $|s|\rightarrow +\infty$, hence
\begin{equation*}
	\int_{-\infty}^{+\infty}|f(i\omega)|^2d\omega < +\infty.
\end{equation*}
Therefore by the Smirnov theorem (see \cite{Nikolski}) $f\in H^2(\mathbb{C}_+)$, where $H^2\left(\mathbb{C}_+\right)$ is the so-called Hardy space, and by the Paley-Wiener theorem (see \cite{Harper, Yosida}) there exists a unique function $F\in L^2\left(\mathbb{R}_+\right)$ such that $\laplace{F}(s)=f(s)$ with 
	\begin{equation*}F(t)= \lim\limits_{v\rightarrow +\infty}\frac{1}{2\pi}\int_{-v}^{v}f(c+i\omega)e^{(c+i\omega)t}d\omega
\end{equation*}is to be understood in the sense of $L^2$ Fourier transforms for any $c>0$. On the other hand, from \eqref{infinitybeh} we have 
$g(s)=f(s)-\frac{1}{4s}$ is Lebesgue integrable on the line $\mathrm{Re} s=c$ for any $c>0$. From Lemma 3.9. of \cite{Rognlie} there exists a function $\widetilde{F}\in C([0,+\infty))$ such that $\laplace{\widetilde{F}}(s)=g(s)$, with 
$$\widetilde{F}(t)= \frac{1}{2\pi}\int_{-\infty}^{+\infty}\left[f(c+i\omega)- \frac{\lambda}{c+i\omega}\right]e^{(c+i\omega) t}d\omega$$ independent of $c>0.$  Hence, writing $f(s)=g(s) + \frac{1}{4s}$ and using the fact that $\laplace{\lambda}=\frac{\lambda}{s}$ for all complex constant $\lambda$, we have that $\laplace{F}(s)=f(s)$ with
$$F(t)= \frac{1}{2\pi}\int_{-\infty}^{+\infty}\left[f(c+i\omega)- \frac{\lambda}{c+i\omega}\right]e^{(c+i\omega) t}d\omega + \lambda$$
and $\lambda=\frac{1}{4}.$
\end{proof}
\noindent Then we can write the coupling term with the fluid motion $\zeta_e(t,R)$ as an explicit function of the solid velocity $\dot{\delta}_G$ under convolution form:
\begin{proposition}\label{linzeta}
	Considering the linearized shallow water equations in the exterior domain, the following holds:
	\begin{equation}\zeta_e(t,R)=\int_{0}^{t}F(s)\dot{\delta}_G(t-s)ds - \dfrac{R}{2v_0}\dot{\delta}_G(t)
	\label{linzetae}
	\end{equation} with $F(t)$ as in Lemma \ref{F}.
\end{proposition}
	\begin{remark}
		It is easy to see that there exists a function $F_0$ such that 
		$$F(t)=F_0\left(\frac{v_0}{R}t\right)$$ with 
		$$\mathcal{L}(F_0)(s)=\dfrac{i H_0^{(1)}\left(is \right)}{2H_1^{(1)}\left(is \right)}+\dfrac{1}{2}=: f_0(s).$$
	\end{remark}
Now we show the long time behaviour of $F$, which due to the previous remark is is independent on the choice of the parameters $v_0$ and $R$.
\begin{proposition}\label{t-2decrease}
		$F$ has a long time polynomially decaying behaviour. In particular there exist $M>0$ such that $$|F(t)|\leq M \,(1+t)^{-2} $$for all $t\geq0.$ 
\end{proposition}
\begin{proof}Without loss of generality we prove the lemma for $F_0$, which is a function independent of the parameters $v_0$ and $R$. As said before, $F(t)=F_0(\frac{v_0}{R}t)$ and the asymptotic behaviour of $F$ is the same as the one of $F_0$.\\First, let us recall that $F_0$ is defined as the inverse Laplace transform of $f_0$ that has a branch cut on the real negative semi-axis by choosing the value of the complex logarithm that has a branch cut in the lower imaginary semi-axis in the series expansion of the Bessel function $Y_n$ for $n=0,1$ in Appendix A. Let us consider the closed curve (see Figure \ref{contour})

	$$\mathcal{C} = \gamma^{-L, -\delta_1}_{-\varepsilon} \vee \gamma^{-\delta_1}_{-\varepsilon, \delta_2} \vee \gamma_{\delta_2}^{-\delta_1, \delta_1} \vee -\gamma^{\delta_1}_{ -\varepsilon, \delta_2} \vee \gamma^{\delta_1,L}_{-\varepsilon}\vee\gamma_{-\varepsilon, c}^{L} \vee -\gamma_{c}^{-L,L}\vee -\gamma_{-\varepsilon, c}^{-L},$$
	where 
	\begin{center}
		\begin{tabular}{r r l }
	$	\gamma_{-\varepsilon}^{-L, -\delta_1}:$&$\quad [-L, -\delta_1]\rightarrow \mathbb{C}, $& $\quad y \mapsto -\epsilon +i y,$\\[7pt]
	$	\gamma^{-\delta_1}_{-\varepsilon, \delta_2}:$&$\quad[-\epsilon, \delta_2]\rightarrow \mathbb{C}, $&$\quad  x\mapsto x -i\delta_1,$\\[7pt]
		$\gamma_{\delta_2}^{-\delta_1, \delta_1}:$&$\quad[-\delta_1, \delta_1]\rightarrow \mathbb{C},  $&$\quad y\mapsto \delta_2 +i y,$\\[7pt]$
			\gamma^{\delta_1}_{-\varepsilon, \delta_2}:$&$\quad[-\epsilon, \delta_2]\rightarrow \mathbb{C}, $&$\quad x\mapsto x +i\delta_1,$\\[7pt]$
			\gamma_{-\varepsilon}^{\delta_1,L}:$&$\quad[\delta_1,L]\rightarrow \mathbb{C},  $&$\quad y \mapsto -\epsilon +i y, $\\[7pt]$
		\gamma_{-\varepsilon, c}^{L}:$&$\quad[-\epsilon, c]\rightarrow \mathbb{C}, $&$\quad x\mapsto x +iL,$\\[7pt]$
		\gamma_{c}^{-L,L}:$&$\quad [-L, L]\rightarrow \mathbb{C},  $&$\quad y \mapsto c +i y,$\\[7pt]$
		\gamma_{-\varepsilon, c}^{-L}:$&$\quad[-\epsilon, c]\rightarrow \mathbb{C},  $&$\quad x\mapsto x -iL,\vspace{1em}$
		\end{tabular}
	\end{center}
		\begin{figure}\centering
		\includegraphics[scale=1]{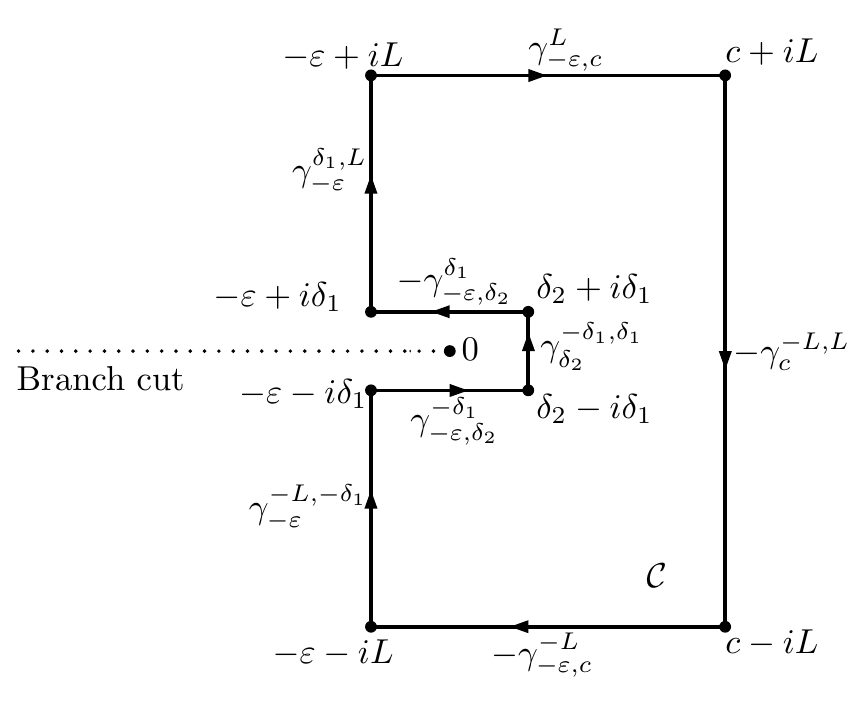}
		\caption{The closed curve  $\mathcal{C}$ in the complex plane.}
		\label{contour}
	\end{figure}
with $\varepsilon,\, \delta_1, \,\delta_2, L >0$ and $c>0$ as in Lemma \ref{F}. We choose $\varepsilon$ small enough such that the zeros of $H_1^{(1)}(is)$ lay in the left of the curve $\mathcal{C}$. This can be done since from \cite{NIST:DLMF} the zeros of $H_1^{(1)}(is)$ all have real part less than $-\alpha$ for some $\alpha>0$. Hence $f_0$ is holomorphic in the interior of  $\mathcal{C}$ and  by Cauchy's integral theorem
	$$\frac{1}{2 \pi i }\int_\mathcal{C} f_0(s)e^{st} ds=0, $$
	or equivalently
	\begin{align*}
\frac{1}{2 \pi i }\int_{\gamma_{c}^{-L,L}} f_0(s)e^{st} ds=\ & \frac{1}{2 \pi i }\int_{\gamma_{-\varepsilon}^{-L, -\delta_1}} f_0(s)e^{st} ds+\frac{1}{2 \pi i }\int_{\gamma^{-\delta_1}_{-\varepsilon, \delta_2}} f_0(s)e^{st} ds\\&+ \frac{1}{2 \pi i }\int_{\gamma_{\delta_2}^{-\delta_1, \delta_1}} f_0(s)e^{st} ds  -\frac{1}{2 \pi i }\int_{\gamma^{\delta_1}_{-\varepsilon, \delta_2}} f_0(s)e^{st} ds\\&+\frac{1}{2 \pi i }\int_{\gamma_{-\varepsilon}^{\delta_1,L}} f_0(s)e^{st} ds + \frac{1}{2 \pi i }\int_{	\gamma_{-\varepsilon, c}^{L} \vee\,-\gamma_{-\varepsilon, c}^{-L} } f_0(s)e^{st} ds.\end{align*}
	Taking the limit $L\rightarrow +\infty$, the limit $\delta_2\rightarrow 0^+$ and the limit $\delta_1\rightarrow 0^+$,  by definition of the inverse Laplace transform as a Bromwich integral we have
	\begin{equation}
	\begin{aligned}
		F_0(t)= &\, \frac{1}{2 \pi i }\lim\limits_{\delta_1\rightarrow 0^+}\bigg[\lim\limits_{L\rightarrow +\infty}\left(\int_{\gamma_{-\varepsilon}^{-L, -\delta_1}}f_0(s)e^{st} ds+ \int_{\gamma_{-\varepsilon}^{\delta_1,L}} f_0(s)e^{st} ds \right)\\&+\lim\limits_{\delta_2\rightarrow 0^+}\bigg(\int_{\gamma^{-\delta_1}_{-\varepsilon, \delta_2}} f_0(s)e^{st} ds+\int_{\gamma_{\delta_2}^{-\delta_1, \delta_1}} f_0(s)e^{st} ds  -\int_{\gamma^{\delta_1}_{-\varepsilon, \delta_2}} f_0(s)e^{st} ds\bigg)\bigg], 
	\end{aligned}
	\label{F_0}
	\end{equation}
	where we have used the fact that \begin{align*}
		&\lim\limits_{L\rightarrow +\infty}\int_{\gamma_{-\varepsilon, c}^{L} \vee\,-\gamma_{-\varepsilon, c}^{-L} } f_0(s)e^{st} ds\\&=\lim\limits_{L\rightarrow +\infty}\left(e^{iLt}\int_{-\epsilon}^c f_0(x+iL)e^{vt}dx-e^{-iLt}\int_{-\epsilon}^c f_0(x-iL)e^{xt}dx\right)=0	\end{align*}by uniform convergence.
	Let us deal with the first two terms in \eqref{F_0}. As we did in Lemma \ref{F}, we write $f_0(s)=g_0(s) + \frac{1}{4s}$ with $g_0(s)$ integrable on the line $\mathrm{Re} (s)=-\varepsilon $. Hence we have that 
	$$\lim\limits_{\delta_1\rightarrow 0^+}\lim\limits_{L\rightarrow +\infty}\left(\int_{\gamma_{-\varepsilon}^{-L, -\delta_1}}g_0(s)e^{st} ds+ \int_{\gamma_{-\varepsilon}^{\delta_1,L}} g_0(s)e^{st} ds \right)$$ is finite and is exponentially decaying in $t$. We focus now on the contour integration of $\frac{1}{4s}$. We compute that
		\begin{align*}
	&\frac{1}{2\pi i }\left(\int_{\gamma_{-\varepsilon}^{-L, -\delta_1}}\frac{e^{st}}{4s} ds+ \int_{\gamma_{-\varepsilon}^{\delta_1,L}} \frac{e^{st}}{4s}ds\right)\\[5pt]& = \frac{e^{-\varepsilon t}}{8\pi}\left(\int_{-L}^{-\delta_1}\frac{e^{iyt}}{-\varepsilon + iy}dy +\int_{\delta_1}^{L}\frac{e^{iyt}}{-\varepsilon + iy}dy\right)=\frac{e^{-\varepsilon t}}{4\pi} \int_{\delta_1}^L\mathrm{Re}\left(\frac{e^{iyt}}{-\varepsilon+iy}\right)dy.
	\end{align*}
Since the integrand is even and it is bounded close to zero, we have that
$$\lim\limits_{\delta_1\rightarrow 0^+}\lim\limits_{L\rightarrow +\infty}\int_{\delta_1}^L\mathrm{Re}\left(\frac{e^{iyt}}{-\varepsilon+iy}\right)dy=\frac{1}{2}\lim\limits_{L\rightarrow +\infty}\int_{-L}^L\mathrm{Re}\left(\frac{e^{iyt}}{-\varepsilon+iy}\right)dy.$$ The integral in the right-hand side can be computed by considering the contour integration of the associated complex-valued function along the semicircle of radius $L$ centered at the origin and closed on the real axis. Since the function is holomorphic, by Cauchy's integral theorem 
$$0=\int_{-L}^L\mathrm{Re}\left(\frac{e^{iyt}}{-\varepsilon+iy}\right)dy + \int_{C_L}\mathrm{Re}\left(\frac{e^{izt}}{-\varepsilon+iz}\right)dz, $$ where $C_L$ is the arch of the semicircle. The second term vanishes as $L\rightarrow +\infty$, hence it yields 
$$\lim\limits_{L\rightarrow +\infty}\int_{-L}^L\mathrm{Re}\left(\frac{e^{iyt}}{-\varepsilon+iy}\right)dy=0.$$
For the last three terms in \eqref{F_0} we have that

	\begin{equation}
		\begin{aligned}
			\frac{1}{2 \pi i }\lim\limits_{\delta_1\rightarrow 0^+}\lim\limits_{\delta_2\rightarrow 0^+}\left(\int_{\gamma^{-\delta_1}_{-\varepsilon, \delta_2}} f_0(s)e^{st} ds+\int_{\gamma_{\delta_2}^{-\delta_1, \delta_1}} f_0(s)e^{st} ds - \int_{\gamma^{\delta_1}_{-\varepsilon, \delta_2}} f_0(s)e^{st} ds\right)\\ = \frac{1}{2 \pi i }\int_{-\varepsilon}^{0} \lim\limits_{\delta_1\rightarrow 0^+}\left(f_0(x-i\delta_1)- f_0(x+i\delta_1) \right)e^{xt} dx, 
		\end{aligned}
	\label{limlim}
	\end{equation}by uniform convergence.
 Using the analytic continuation formulas in Appendix A, we compute that
		\begin{align*}
			\lim\limits_{\delta_1\rightarrow 0^+}\left(f_0(x-i\delta_1)- f_0(x+i\delta_1) \right)&= \frac{i}{2}\left[\frac{3H_0^{(1)}(ix) +2H_0^{(2)}(ix)}{3H_1^{(1)}(ix) +2H_1^{(2)}(ix)}-\frac{H_0^{(1)}(ix)}{H_1^{(1)}(ix)}\right]\\[10pt]
			& =  -2\,\frac{\,J_0(ix)Y_1(ix)-Y_0(ix)J_1(ix)}{(5J_1(ix)+iY_1(ix))(J_1(ix)+iY_1(ix))},
		\end{align*}
		and using the series expansion of $J_n$ and $Y_n$ for $n=0,1$ in Appendix A we get
		\begin{equation}
			\lim\limits_{\delta_1\rightarrow 0^+}\left(f_0(x-i\delta_1)- f_0(x+i\delta_1) \right)= -\pi i x + O(x^2) \quad \,\mbox{as} \quad x\rightarrow 0. \label{diffbranch}
		\end{equation}Therefore in \eqref{limlim} we have 
		\begin{equation*}
		\begin{aligned}
		&	\frac{1}{2 \pi i }\int_{-\varepsilon}^{0}\lim\limits_{\delta_1\rightarrow 0^+}\left(f_0(x-i\delta_1)- f_0(x+i\delta_1) \right)e^{xt} dx\\&=-\frac{1}{2}\int_{-\varepsilon}^0 x e^{xt} dx + \int_{-\varepsilon}^0 O(x^2)e^{xt} dx\\&
			=\frac{1}{2t^2}\int_0^{\varepsilon t}\sigma e^{-\sigma}d\sigma + \int_0^\varepsilon O(x^2)e^{-xt}dx\,=\ \frac{1}{2t^2} + O(t^{-3}),
			\end{aligned}
		\end{equation*}
		which implies the statement of the lemma.
\end{proof}
	\begin{figure}\centering
		\includegraphics[scale=0.5]{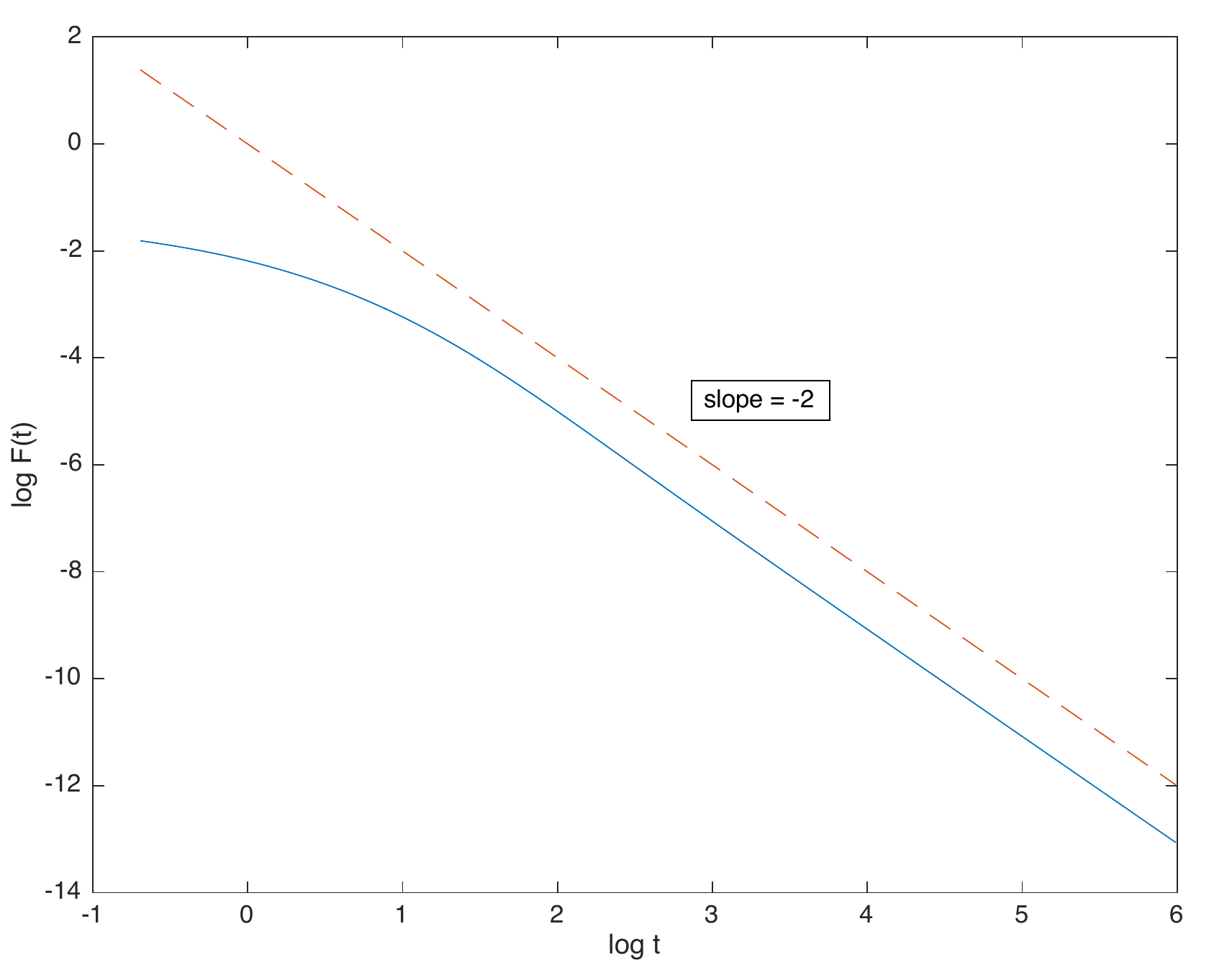}
		\caption{Polynomial decay of $F$: here $F(t)$ (full) is compared to $t^{-2}$ (dash) using two logarithmic scales, for $R=10\,\textrm{m}$ and $h_0=5\,\textrm{m}$.}
		\label{plotF}
	\end{figure}
	
The polynomial decay of the impulse response function $F$ given by Proposition \ref{t-2decrease} is numerically showed in Figure \ref{plotF} for a particular set of parameters.
	Moreover, the kernel $F$ satisfies the following equality, which will be used in the proof of the Theorem \ref{globalex}:
	
	\begin{lemma}
		The convolution kernel $F$ is such that
		\begin{equation}
		\int_{0}^{+\infty}F(t)dt =\frac{R}{2v_0}.
		\label{propertyF}
		\end{equation}

	\end{lemma}
	\begin{proof}
		By the definition of the Laplace transform and by Lemma \ref{F}, 
		\begin{equation}
		\int_{0}^{+\infty}F(t)e^{-st}dt= \dfrac{iR H_0^{(1)}\left(\dfrac{is R}{v_0}\right)}{2v_0H_1^{(1)}\left(\dfrac{is R}{v_0}\right)}+\dfrac{R}{2v_0}\quad\mbox{for}\quad \mathrm{Re} (s)>0.
		\label{laptraF}
		\end{equation}
		From Appendix A, we have that, as $s\rightarrow 0$, $$\dfrac{H_0^{(1)}(is)}{H_1^{(1)}(is)} \sim -is\log(is)\rightarrow 0.$$
		Hence, taking the limit $s\rightarrow 0^+$ in \eqref{laptraF} we get 
		\begin{equation}
		\int_{0}^{+\infty}F(t)dt= \dfrac{R}{2v_0},
		\label{integralF}
		\end{equation}where we have used Lebesgue's dominated convergence theorem due to Lemma \ref{F} and Proposition \ref{t-2decrease}.
	\end{proof}



\subsection{Integro-differential equation for the solid motion}
 From now on we suppose for simplicity that the bottom of the structure is flat, then $\zeta_w$ (as well as $h_w$) does not depend on the space variable $r$, but Proposition \ref{nonlinearcummins} holds for a structure with non-flat bottom as well.
We know from Proposition \ref{linzeta} that, considering the linear shallow water equations \eqref{linearmodel} in the exterior domain, we can write the trace of the surface elevation $\zeta_e$ at the boudary $r=R$ as a function of the time derivative of the displacement $\delta_G.$ Then the nonlinear differential equation \eqref{eqmotionIntro} describing the solid motion can be written as a nonlinear delay differential equation.

\begin{proposition}\label{nonlinearcummins}
	Considering the linear shallow water equations \eqref{extlinshallow} for the fluid motion in the exterior domain, the solid motion is described by the following second order nonlinear integro-differential equation:
	\begin{equation}\begin{aligned}
	(m+m_a(\delta_G)) \ddot{\delta}_G=& -\mathfrak{c}\delta_G-\nu\dot{\delta}_G+\mathfrak{c} \int_{0}^{t}F(s)\dot{\delta}_G(t-s)ds +\left(\mathfrak{b}(\dot{\delta}_G)+\beta(\delta_G)\right)\dot{\delta}^2_G \ ,
	\end{aligned}
	\label{expliciteqmotion}
	\end{equation}with $\mathfrak{c}$ as in \eqref{ODEparam}, $\nu=\dfrac{\mathfrak{c}R}{2v_0}$, $m_a(\delta_G)=\dfrac{b}{h_w(\delta_G)},$ $\beta(\delta_G)=\dfrac{b}{2h^2_w(\delta_G)}.$ The convolution kernel is defined by
	\begin{equation*} \displaystyle
	F(t)=\lim\limits_{v\rightarrow +\infty} \frac{1}{2\pi }\bigint_{-v}^{v}\left(\dfrac{iR H_0^{(1)}\left(\dfrac{i(c +i\omega) R}{v_0}\right)}{2v_0H_1^{(1)}\left(\dfrac{i(c +i\omega) R}{v_0}\right)}+\dfrac{R}{2v_0}\right)e^{(c+i\omega)t}dw\end{equation*}for any $c>0$ and 
	\begin{equation*}\mathfrak{b}(\dot{\delta}_G)=\dfrac{b}{\left(\int_{0}^{t}F(s)\dot{\delta}_G(t-s)ds -\frac{R}{2v_0}\dot{\delta}_G(t)+h_0\right)^2} \ \ \ \ \ \mbox{with} \ \ \ \ b=\frac{\pi \rho R^4}{8}.
	\end{equation*}
	\end{proposition}
	\begin{remark}
	 In the integro-differential equation \eqref{expliciteqmotion} $m_a(\delta_G)$ is the time dependent added mass, $\mathfrak{c}$ is the hydrostatic coefficient and $\nu$ is the damping coefficient. The convolution term, whose kernel $F$ is the so-called impulse response (or retardation) function, accounts for fluid-memory effects that incorporate the energy dissipation due to the radiated waves coming from the motion of the structure. Moreover, linearizing \eqref{expliciteqmotion} around the equilibrium state, we get 
	 	\begin{equation}
	 	\left(m+m_a(0)\right)\ddot{\delta}_G(t)= -\mathfrak{c}\delta_G(t)-\nu\dot{\delta}_G(t)+\mathfrak{c} \int_{0}^{t}F(s)\dot{\delta}_G(t-s)ds.
	 	\label{cummins}
	 	\end{equation}This linear equation is nothing but the well-known Cummins equation for the vertical displacement \eqref{integrodiff}. Proposition \ref{nonlinearcummins} therefore provides a rigorous justification of the Cummins equation and generalizes it to take into account the nonlinear effects in the interior domain.
	\end{remark}

\begin{remark}Recall that in Proposition \ref{linzeta} we show that 
	\begin{equation}\zeta_e(t,R)=\int_{0}^{t}F(s)\dot{\delta}_G(t-s)ds - \dfrac{R}{2v_0}\dot{\delta}_G(t).
	\end{equation}
	Therefore, considering the linear equations \eqref{linearmodel}, the extension-trace operator \eqref{extensiontrace} becomes a linear convolution operator, that is
	\begin{equation}
		\mathcal{B} \left[\dot{\delta}_G\right](t)= \int_{0}^{t}K(s)\dot{\delta}_G(t-s)ds ,
	\end{equation}with the convolution kernel $K(s)=F(s)- \dfrac{R}{2v_0}\mathfrak{d}_{s=0},$ where $\mathfrak{d}_{s=0}$ is the Dirac delta distribution.	
	\end{remark}

	
	We state now the following global existence and uniqueness result of the solution to the solid motion equation in the case of linear shallow water equations for the fluid motion. Let us recall that $\rho_m$ and $H$ denote the density and the height of the solid, respectively. 		
\begin{theorem}The Cauchy problem for the nonlinear second order integro-differential equation
 \eqref{expliciteqmotion} with initial data
 \begin{equation*}\delta_G(0)=\delta_0 \neq0, \ \ \ \ \  \ \ \ \ \ \ \dot{\delta}_G(0)=0,\end{equation*} admits a unique solution $\delta_G\in C^2([0, +\infty), \mathbb{R})$ provided 
 \begin{equation}
 |\delta_0 |< \min \left(h_0 -\frac{\rho_m H}{\rho}, \sqrt{\frac{\rho_m H}{\rho  g}}\frac{h_0}{\|F\|_{L^1(\mathbb{R}_+)}+\frac{R}{2v_0}}\right).
 	\label{admissibility}
 \end{equation}
 \label{globalex}
\end{theorem}
\begin{remark}
		One needs to consider the parameters of the problem such that 
		$$h_{w,eq}=h_0 - \frac{\rho_m H}{\rho}>0,$$
		which means that the fluid height under the solid at the equilibrium position is positive.
	\end{remark}
 
\begin{proof}Let us write \eqref{expliciteqmotion} as the following first order nonlinear integro-differential equation on $x(t)=(\delta_G(t), \dot{\delta}_G(t))^T$: \begin{equation}\begin{cases}
	\dfrac{d x(t)}{dt}= \mathcal{F}\left(t, x(t), \int_0^t K(t,s) x(s)ds\right),   \\[10pt]
	x(0)=(\delta_0,0)^T,
	\end{cases}
	\label{firstorder}\end{equation} with $K(t,s)=F(t-s).$
 The local existence and uniqueness of the solution to \eqref{firstorder} can be obtained using an iterative method in the Banach space $C([0,T), \mathbb{R}^2)$.
 This integro-differential equation can also be seen as a functional differential equation with infinite delay by extending $\delta_G$ and $\dot{\delta}_G$ respectively by $\delta_0$ and by $0$ for $t<0$. We refer to Liu and Magal \cite{LiuMag} for the analysis of this type of delayed equations, which would give asymptotic stability of the equilibrium position in the case of an exponentially decaying convolution kernel.\\
 We show here the extension of the existence interval $[0,T)$ to $[0,+\infty)$ using the conservation of total fluid-structure energy. 
 First, we prove in the following lemma that the solution is bounded.
	\begin{lemma}\label{bound}
		The displacement $\delta_G$ and its derivative $\dot{\delta}_G$ are both bounded.
	\end{lemma}
	\begin{proof}
		From Proposition \ref{conservationenergy} we know that the energy of the coupled floating structure system considering the linear shallow water equations for the fluid motion 
		\begin{equation}
		E_{tot}(t)=\dfrac{1}{2}m\dot{\delta}_G^2(t) + mg\delta_G(t)+ E_{SW}(t)
		\end{equation}is conserved. 
		Moreover, $E_{SW}(t)$ can be written as the sum of the fluid energy in the interior domain, 
		
		\begin{equation*}E_{int}(t)=\dfrac{1}{2}\rho g \,2\pi\int_0^R\zeta^2_w(t)rdr-\dfrac{1}{2}\rho g\,2\pi\int_0^R\zeta^2_{w,eq}rdr +\dfrac{1}{2}\rho\,2\pi\int_{0}^R\dfrac{q_i^2(t,r)}{h_w(t)}rdr,\end{equation*}and the fluid energy in the exterior domain, $$E_{ext}(t)=\dfrac{1}{2}\rho g 2\pi\int_R^{+\infty}\zeta^2_e(t,r)rdr + \dfrac{1}{2}\dfrac{\rho}{h_0}\,2\pi\int_{R}^{+\infty}q^2_e(t,r)rdr.$$To get the expression of the fluid energy in the interior domain we use the constraint \eqref{interiorconstraint} and we add the constant term $\frac{1}{2}\rho g\, 2\pi\int_0^R\zeta^2_{w,eq}r dr$ in order to have zero energy at the equilibrium.
		From Archimedes' principle we have\begin{equation}-\rho_m H -\rho \zeta_{w,eq}=0
		\label{zetaWeq}
		\end{equation} and, since the bottom of the solid is flat, we have $$z_{G,eq}=\zeta_{w,eq}+\dfrac{H}{2}.$$ Then \begin{equation*}z_{G,eq}=\left(\dfrac{1}{2}-\dfrac{\rho_m}{\rho}\right)H
		\label{zgeq}
		\end{equation*} and 
		\begin{equation}
		\zeta_w(t)=z_G(t)-\dfrac{H}{2}=\delta_G(t)+z_{G,eq}-\dfrac{H}{2}=\delta_G(t)-\dfrac{\rho_m H}{\rho}.
		\label{zetaWrho}
		\end{equation}
		Using \eqref{zetaWrho} and the fact that $q_i(t,r)=-\dfrac{r}{2}\dot{\delta}_G(t)$ (see Section \ref{FSequations}),
		the fluid energy in the interior domain $E_{int}(t)$ becomes 
		$$E_{int}(t)=\dfrac{1}{2}g \rho\pi R^2\left(\delta_G(t)-\dfrac{\rho_m H}{\rho}\right)^2-\dfrac{1}{2}g\rho\pi R^2\dfrac{\rho^2_m H^2}{\rho^2} + \dfrac{\pi\rho R^4}{16h_w(t)}\dot{\delta}^2 _G(t).$$
		In particular the total energy at instant $t=0$ is  $$E_{tot}(0)=mg\delta_0+\dfrac{1}{2}g\rho \pi R^2\left(\delta_0 -\dfrac{\rho_m H}{\rho}\right)^2-\dfrac{1}{2}g\rho\pi R^2\dfrac{\rho^2_m H^2}{\rho^2},$$ using $\delta_G(0)=\delta_0$ and $\dot{\delta}_G(0)=0.$ By the conservation of total energy and using the identity $m=\rho_m \pi R^2 H$ we have\begin{equation} \begin{aligned}
		\left(\frac{m}{2}+\frac{\pi\rho R^4}{16 h_w(t)}\right)\dot{\delta}_G^2(t)=& \ mg\delta_0+\frac{1}{2}g\rho \pi R^2\left(\delta_0 -\frac{\rho_m H}{\rho}\right)^2-mg\delta_G(t)
		\\&-\frac{1}{2}g \rho\pi R^2\left(\delta_G(t) -\frac{\rho_m H}{\rho}\right)^2 - E_{ext}(t)\\
		=&\ \frac{1}{2} g \rho \pi R^2 (\delta_0^2 - \delta_G^2(t)) - E_{ext}(t).
		\end{aligned}
		\label{energycons}
		\end{equation}
		Consider $t^*=\sup\{t\in[0, T )\ | \ h_w(s)>0 \mbox{ for} \ s\in (0, t) \}$. From condition \eqref{admissibility} we have $h_w(0)=h_{w,eq}+\delta_0>0$ , hence $t^* >0$. Suppose $t^*<T$. Then, for $t\in(0, t^*)$ the right-hand side of \eqref{energycons} has to be non-negative. By solving the inequality with respect to $\delta_G(t)$, we have 
		$$-\sqrt{\delta_0^2 -\dfrac{2E_{ext}(t)}{g\rho\pi R^2}}\leq\delta_G(t)\leq\sqrt{\delta_0^2 -\dfrac{2E_{ext}(t)}{g\rho\pi R^2}}.$$
		By the non-negativity of $E_{ext}(t)$ we get the bound \begin{equation} -|\delta_0|\leq \delta_{G}(t)\leq |\delta_0|.
		\label{bounddelta}
		\end{equation}
 Moreover, from \eqref{energycons} we have 
\begin{equation*}
	\begin{aligned}
	\frac{m}{2}\dot{\delta}^2_G \leq \frac{1}{2} g \rho \pi R^2 \delta_0^2,
	\end{aligned}
\end{equation*}
which yields the bound 
		\begin{equation}
		|\dot{\delta}_G(t)|\leq \sqrt{\frac{\rho g}{\rho_m H}}|\delta_0|.
		\label{bounddeltadot}
		\end{equation}
		Using condition \eqref{admissibility}, by continuity we have $h_w(t^*)\geq h_{w,eq}-|\delta_0| >0$ and there exists $\epsilon>0$ small enough such that $h_w(t^* + \epsilon)>0$, where $t^*$ is the maximal time such that $h_w(t)>0$ for $t\in(0, T)$. Then necessarily $t^*= T,$ which implies that the bounds \eqref{bounddelta} - \eqref{bounddeltadot} hold in the existence interval $(0, T)$.
	\end{proof}
	The bounds \eqref{bounddelta} - \eqref{bounddeltadot} give
	\begin{equation*}
	h_w(t)= h_{w,eq} + \delta_G(t) \geq h_{w,eq} - |\delta_0|,
	\end{equation*}
	\begin{equation*}
	\begin{aligned}
	h_e(t,R)&=\int_{-\infty}^{0}F(-\theta)\dot{\delta}_G(t+\theta)d\theta - \frac{R}{2v_0}\dot{\delta}_G(t) + h_0\\[5pt]& \geq  -\left(\|F\|_{L^1(\mathbb{R}_+)}+ \frac{R}{2v_0}\right)\sqrt{\frac{\rho g}{\rho_m H}}|\delta_0| + h_0. 
	\end{aligned}
	\end{equation*}
	We remark that $F\in L^1(\mathbb{R}_+)$ due to Lemma \ref{F} and Proposition \ref{t-2decrease}.
The admissibility condition \eqref{admissibility} on $\delta_0$ guarantees that for all $t\geq 0$ \begin{equation}h_{w}(t)\geq h_{w,eq} -|\delta_0|>0,
	\label{hwmin}\end{equation} \begin{equation}h_e(t,R)\geq -\left(\|F\|_{L^1(\mathbb{R}_+)}+ \frac{R}{2v_0}\right)\sqrt{\frac{\rho g}{\rho_m H}}|\delta_0| + h_0 >0.\label{hemin}\end{equation} 
	 Therefore, the existence interval can be extended to $[0,+\infty)$ by iterating the previous argument.
\end{proof}
\begin{remark}
	The conditions \eqref{hwmin} - \eqref{hemin} express the physical fact that, during all the motion, both the solid and the fluid trace at the solid walls do not touch the bottom of domain.
\end{remark}
\section{Numerical method} 
\begin{figure}\centering
	\includegraphics[scale=0.75]{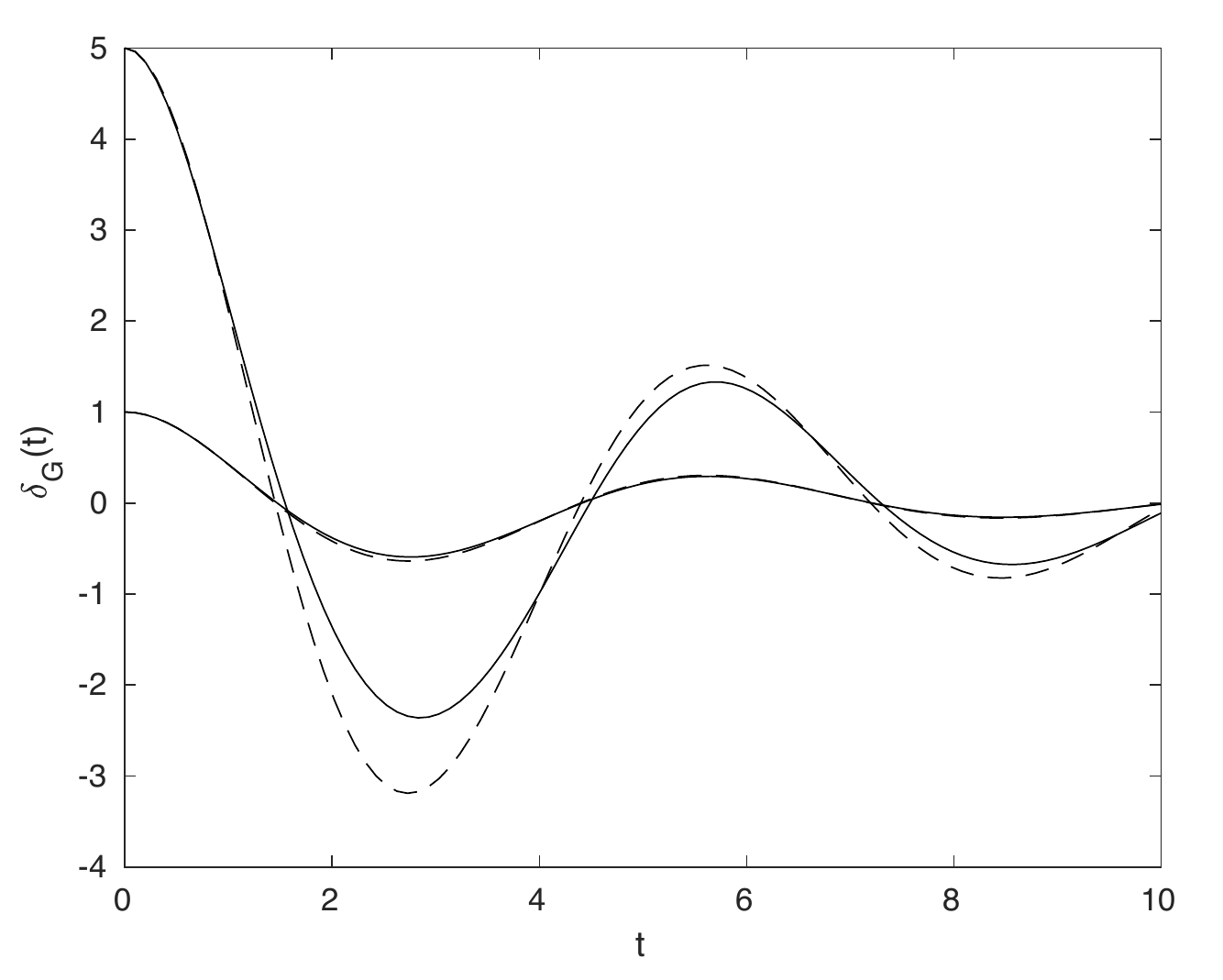}
	\caption{Time evolution of the displacement $\delta_G$ given by the nonlinear integro-differential \eqref{expliciteqmotion} (full) and by the linear Cummins equation \eqref{cummins} (dash) for two different initial data.}
	\label{compODEmotion}
\end{figure}
In order to solve numerically the delay differential equation \eqref{expliciteqmotion}  we write it under the form 
$$\frac{dy}{dt}(t)=f\left(t,y(t),y(d_1(t)),...,y(d_k(t))\right),$$
with $d_1(t),...,d_k(t)$ the components of the non-constant delays vector $d(t)$. In our case we have chosen $d_k(t)=t- k dt$ with $dt=0.1$ and $k=1,...,N$ for $N=100$. Then, we implement in our code the MATLAB solver \texttt{ddesd}, which integrates with the explicit Runge-Kutta (2,3) pair and interpolant of \texttt{ode23}. For more details on the solver we refer to Shampine \cite{Shampine}. Moreover, we compute the convolution integral applying the trapezoidal integration method following Armesto et al. \cite{Armesto}.
In an analogous way, we compute the convolution kernel $F$ for a given set of time steps $n\Delta t$ with $n=1,...,N$ since the influence of the Kernel is negligible after some time $t^*=N\Delta t$.
Then, we compare the numerical result given by the nonlinear integro-differential equation \eqref{expliciteqmotion} with the one obtained from its linear approximation. In Figure \ref{compODEmotion} we consider $h_0=15 \ \text{m}$, $R=10 \ \text{m}$, $H=10 \ \text{m}$, $\rho=1000 \ \text{kg}/\text{m}^3$ and the volume density of the solid $\rho_m=0.5 \ \rho.$ We choose two different initial data: $\delta_0=1 \ \text{m}$ and $\delta_0=5 \ \text{m}$. One can see that for large amplitudes the nonlinear effects should not be neglected in order to better describe the solid motion. This difference justifies the approach to keep nonlinearities in the equation of the floating body problem in the interior domain. Moreover,one can note that the displacement goes to zero but the structure definitely does not reach its equilibrium position: this is due to the motion of the fluid which makes the solid constantly move.

\appendix

\section{Hankel functions}
In this appendix we show some results and properties for the Hankel functions. 
Let us consider the following differential equation:
\[z^{2}\dfrac{{\mathrm{d}}^{2}w}{{\mathrm{d}z}^{2}}+z\dfrac{\mathrm{d}w}{\mathrm{d%
	}z}+(z^{2}-\nu^{2})w=0, \ \ \ z\in\mathbb{C}.\]
This differential equation is called Bessel equation of index $\nu$. Solutions to this equation are called Bessel functions. Let us consider the case when $\nu=n$, with $n\in \mathbb{Z}$. Bessel functions of the first kind, denoted by $J_n(z)$,

\[\mathop{J_n\/}\nolimits\!\left(z\right)=(\tfrac{1}{2}z)^n\sum_{k=0}^{%
	\infty}(-1)^{k}\dfrac{(\tfrac{1}{4}z^{2})^{k}}{k!\mathop{\Gamma\/}\nolimits\!%
	\left(n+k+1\right)}.\]
\label{serieJ}
are entire in z.\\\\
Bessel functions of the second kind, denoted by $Y_{n}(z)$
\begin{align*}\mathop{Y_{n}\/}\nolimits\!\left(z\right)=&-\dfrac{(\tfrac{1}{2}z)^{-n}}{\pi}%
	\sum_{k=0}^{n-1}\dfrac{(n-k-1)!}{k!}\left(\tfrac{1}{4}z^{2}\right)^{k}+\dfrac{2}%
	{\pi}\mathop{\log\/}\nolimits\!\left(\tfrac{1}{2}z\right)\mathop{J_{n}\/}%
	\nolimits\!\left(z\right)\\&-\dfrac{(\tfrac{1}{2}z)^{n}}{\pi}\sum_{k=0}^{\infty}(%
	\mathop{\psi\/}\nolimits\!\left(k+1\right)+\mathop{\psi\/}\nolimits\!\left(n+k%
	+1\right))\dfrac{(-\tfrac{1}{4}z^{2})^{k}}{k!(n+k)!},\end{align*}
where $\psi=\dfrac{\Gamma^\prime}{\Gamma}$, with $\Gamma$ the Gamma function, have a branch point in $z=0$. Both $J_n$ and $Y_n$ are real valued if $z$ is real. 
Let us define $$H_n^{(1)}(z):=J_n(z)+i Y_n(z),$$ $$H_n^{(2)}(z):=J_n(z)-i Y_n(z).$$ We call them respectively $Hankel \mbox{ } functions$ of first order and second order with index $n$, and they are solutions to the Bessel equation. Each solution has a branch point at $z=0$ for all $n$. The principal branches of $H_n^{(1)}(z)$ and $H_n^{(2)}(z)$ are two-valued and discontinuous on the cut along the negative real axis. They are holomorphic functions of $z$ throughout the complex plane cut (see Chapter 9 of \cite{AbraStegun}).\\ 
Now let us show some representations of these functions useful for our problem. From \cite{Treat} we have an integral representation for $z=x>0$:
$$H_n^{(1)}(x)=\dfrac{2e^{-n\pi i/2}}{\pi i}\int_{0}^{+\infty}e^{ix\cosh(s)}\cosh(ns)ds$$ and $$H_n^{(2)}(x)=-\dfrac{2e^{n\pi i/2}}{\pi i}\int_{0}^{+\infty} e^{-ix\cosh(s)}\cosh(ns)ds,$$ 
and a series representation for large $|z|$ and $0 <\arg z < \pi$:
$$H_n^{(1)}(z)=\sqrt{\dfrac{2}{\pi z}}e^{i\left(z - \frac{\pi}{4}- n\frac{\pi}{2}\right)}\bigg[\sum_{k=0}^{p-1} \dfrac{(-)^k a_k(n)}{z^k} + O(z^{-p})\bigg]$$
$$H_n^{(2)}(z)=\sqrt{\dfrac{2}{\pi z}}e^{-i\left(z - \frac{\pi}{4}-  n\frac{\pi}{2}\right)}\bigg[\sum_{k=0}^{p-1} \dfrac{a_k(n)}{z^k} + O(z^{-p})\bigg]$$ with $$a_0(n)=1,$$ $$ a_k(n)=\dfrac{\{4n^2 -1^2\}\{4n^2-3^2\}\cdots \{4n^2-(2k-1)^2)\}}{8^k k! (i)^k}, \ \ \ \ k>0.$$
Last we recall analytic continuation formulas for $m \in \mathbb{Z}$ (see \cite{NIST:DLMF}):
\[\mathop{{H^{(1)}_{n}}\/}\nolimits\!\left(ze^{m\pi i}\right)=(-1)^{mn-1}((m-1)%
\mathop{{H^{(1)}_{n}}\/}\nolimits\!\left(z\right)+m\mathop{{H^{(2)}_{n}}\/}%
\nolimits\!\left(z\right)),\]
\[\mathop{{H^{(2)}_{n}}\/}\nolimits\!\left(ze^{m\pi i}\right)=(-1)^{mn}(m\mathop%
{{H^{(1)}_{n}}\/}\nolimits\!\left(z\right)+(m+1)\mathop{{H^{(2)}_{n}}\/}%
\nolimits\!\left(z\right)).\]
$$\displaystyle\mathop{{H^{(1)}_{n}}\/}\nolimits\!\left(\overline{z}\right)=%
\overline{\mathop{{H^{(2)}_{n}}\/}\nolimits\!\left(z\right)}, \ \ \ \
\displaystyle\hskip 10.0pt\mathop{{H^{(2)}_{n}}\/}\nolimits\!\left(\overline%
{z}\right)=\overline{\mathop{{H^{(1)}_{n}}\/}\nolimits\!\left(z\right)}.$$


\bibliographystyle{siam}
\bibliography{bibliography}

\end{document}